\documentclass[12pt]{amsart}

\usepackage{latexsym}
\usepackage{amsmath}
\usepackage{stmaryrd,epsfig}
\usepackage{xypic}
\usepackage{amssymb}
\usepackage{amscd}
\usepackage{multicol}
\usepackage[cmtip,matrix,arrow]{xy}
\numberwithin{equation}{section}

\newtheorem{theorem}{\bf Theorem}[section]
\newtheorem{corollary}[theorem]{\bf Corollary}
\newtheorem{proposition}[theorem]{\bf Proposition}
\newtheorem{lemma}[theorem]{\bf Lemma}

\newtheorem{definition-theorem}[theorem]{\bf Theorem-Definition}
\newtheorem{remark}[theorem]{\bf Remark}

\usepackage{amssymb}
\usepackage{graphics}
\usepackage{graphicx}
\usepackage{latexsym}
\usepackage{amsmath}
\usepackage{amssymb}
\usepackage{amscd}
\usepackage{multicol}
\usepackage{leftidx}
\usepackage{tikz}

\numberwithin{equation}{section}

\newcommand{\bZ}{\mathbb{Z}}
\newcommand{\bR}{\mathbb{R}}
\newcommand{\bC}{\mathbb{C}}

\newcommand{\g}{\mathfrak{g}}

\newcommand{\mft}{\mathfrak{t}}
\newcommand{\h}{\mathfrak{h}}

\newcommand{\af}{{\rm aff}}
\newcommand{\eve}{{\rm ev}}

\newcommand{\cU}{\mathcal{U}}

\newcommand{\ve}{\varepsilon}

\newcommand{\n}{\mathfrak{n}}
\newcommand{\mfa}{\mathfrak{a}}
\newcommand{\mfu}{\mathfrak{u}}
\newcommand{\mfb}{\mathfrak{b}}

\address{Department of Mathematics and Statistics\\ University of Regina\\ Regina, Saskatchewan \\Canada S4S 0A2}
\email[]{mareal@math.uregina.ca}

\begin{document}
\title[The periodic  quantum Toda lattice]{On the complete integrability of the  periodic  quantum  Toda lattice}

\author{Augustin-Liviu Mare}

\subjclass[2010]{Primary 37K10; Secondary 17B35, 37N20}
\keywords{periodic quantum Toda lattice; complete integrability; Kac-Moody Lie algebras;
foldings of Dynkin diagrams}

\begin{abstract} We prove that the periodic quantum Toda lattice corresponding to any extended Dynkin
diagram is completely integrable. This has been conjectured and  { proved in all classical cases and $E_6$} by Goodman and Wallach at the beginning of the 1980's. 
{ As a direct application, in the context of quantum cohomology of affine flag manifolds, results that were known to hold only for some particular Lie types can now be extended to all types.}
\end{abstract}
\maketitle

\section{Statement of the main result}\label{first}  We first recall the notion of $``ax+b"$ Lie algebra, as defined by Goodman and Wallach
in \cite{Go-Wa1}. (In fact, the definition in { that} paper is more general:
we only present here a special case which is relevant for our goal, { which is to establish the
complete integrability of the periodic quantum Toda lattice}.)
  Let $\mfa$ and $\mfu$ be two finite dimensional real vector spaces
and endow $\mfb:=\mfa\oplus \mfu$  with a Lie bracket $[\cdot, \cdot]$ and an inner product 
$\langle \cdot, \cdot \rangle$ such that:
\begin{itemize}
\item[1.] $\mfa$ is orthogonal to $\mfu$; 
\item[2.] $\mfa$ is a  Lie subalgebra of $\mfb$, $\mfu$ is an ideal of $\mfb$, and
 both $\mfa$ and $\mfu$ are commutative Lie algebras;
 \item[3.] for each $H\in \mfa$, ${\rm ad}(H)|_\mfu$ is a self-adjoint endomorphism of $\mfu$;
 \item[4.] one has the eigenspace decomposition  
$\mfu=\oplus_{\alpha \in \Psi} \mfu_\alpha$,  where 
$$\mfu_\alpha :=\{X\in \mfu\mid [H, X]=\alpha(H)X \  {\rm for \ all \ } H \in \mfa\},$$for some finite subset $\Psi\subset \mfa^*$,
such that   
 $\mfu_\alpha$ are all one-dimensional vector spaces;
\item[5.] there exists an irreducible root system $\Pi\subset \mfa^*$, a system of simple roots $\Delta=\{\alpha_1,\ldots, \alpha_r\}\subset \Pi$, and a 
dominant root $\beta\in \Pi$ such that  $\Psi = \Delta \cup \{-\beta\}$.
\end{itemize}  

Recall that, by definition, a root $\beta$ is {\it dominant} if $\langle \beta, \alpha_i \rangle \ge0$ for all $1\le i \le r$.
If the root system $\Pi$ is simply laced, there exists a unique dominant root, which is the highest root.
If $\Pi$ is of one of the types $B_r, C_r, F_4$, or $G_2$, then there are two dominant roots: one is long (i.e., the highest root)
 and the other  is short. We refer to \cite[Fig.~5.1]{Go-Wa1} for the precise description. 
 The Lie algebra $\mfb$ is essentially determined  by the root system $\Pi$, the simple roots $\Delta$,  and the dominant root $\beta$, see   
 \cite[p.~362]{Go-Wa1}. Let $U({\mfb})$ be the universal enveloping algebra of $\mfb$, which comes equipped 
 with a canonical grading.  
If $Z_1,\ldots, Z_{2r+1}$ is a basis of $\mfb$ and $(c^{ij})_{1\le i,j\le 2r+1}$ the matrix inverse to 
$(\langle Z_i, Z_j\rangle)_{1\le i,j\le 2r+1}$ then the element 
 $$\Omega:=\sum_{i,j=1}^{2r+1} c^{ij}Z_iZ_j$$
of  $U_2(\mfb)$   
 is independent of the choice of the basis and is called the {\it Laplacian} of $\mfb$, 
 see, e.g., \cite[Proposition 1.3.1]{Ap}. 
 
Denote $\alpha_0:=-\beta$ and consider an orthonormal basis $X_0, X_1,\ldots, X_r$ of $\mfu$ such that
$X_i \in \mfu_{\alpha_i}$, for all $0\le i \le r$.  Also consider $H_1, \ldots, H_r \in \mfa$ such that
$$\alpha_i(H_j)=\delta_{ij}, \ {\rm for \ all \ } 1\le i, j \le r.$$
One has  
\begin{equation}\label{laplace}
\Omega= \sum_{i,j=1}^r \langle \alpha_i,\alpha_j\rangle H_i H_j +\sum_{i=0}^rX_i^2,
\end{equation}
where $\langle \cdot, \cdot \rangle$ is the inner product on $\mfa^*$ induced by the canonical isomorphism $\mfa^*\simeq \mfa$. By the PBW theorem, $U(\mfb)$ has a basis given by monomials of the form $X_{{i_1}} \cdots X_{{i_k}} H_{j_1} \cdots H_{j_\ell}$ where $0 \le i_1, \ldots , i_k \le r$ and $1 \le j_1 , \ldots , j_\ell \le r$. Let $U(\mfb)_\eve$ be the linear subspace of $U(\mfb)$ spanned by the monomials of the form $X_{{i_1}}^2 \cdots X_{{i_k}}^2 H_{j_1} \cdots H_{j_\ell}$ where $0 \le i_1, \ldots , i_k \le r$ and $1 \le j_1 , \ldots , j_\ell \le r$.
The projection $\mfb \to \mfa$ is a Lie algebra homomorphism and induces an algebra homomorphism $\mu: U(\mfb) \to U(\mfa) = {\rm Sym}(\mfa)$ called the {\em symbol map}. 
Let $W$ be the Weyl group of $\Pi$, which acts on $\mfa$, and consequently also on the symmetric algebra ${\rm Sym}(\mfa)$.
Choose fundamental homogeneous generators $u_1, \ldots, u_r$ of
${\rm Sym}(\mfa)^{{W}}$; by convention  
 take $u_1:=\sum_{i,j=1}^r \langle\alpha_i, \alpha_j\rangle H_i H_j$.  
  
The goal of this note is to prove the following result, which was conjectured by Goodman and Wallach in \cite{Go-Wa1}.
(We say that $\Gamma \in U(\mfb)$ {\it has degree} $n$ if $\Gamma \in U_n(\mfb) \setminus U_{n-1}(\mfb)$, where
$U_0(\mfb) \subset U_1(\mfb) \subset \ldots$ is the canonical filtration of the universal enveloping algebra.)

\begin{theorem}\label{thm:main} For any $i\in \{1,\ldots, r \}$ there exists a unique 
$\Omega_i \in U(\mfb)_{\eve}$ such that:

\begin{itemize} 
\item $[\Omega_i, \Omega]=0$;

\item $\Omega_i$ has degree equal to $\deg u_i$;

\item the symbol of $\Omega_i$ is  $u_i $,  the chosen  generator of ${\rm Sym}(\mfa)^W$.
\end{itemize} 

Furthermore, for any $i, j\in \{1, \ldots, r\}$ one has $[\Omega_i, \Omega_j]=0$, where $\Omega_1:=\Omega$.  
\end{theorem} 


{ The theorem was proved in \cite{Go-Wa1, Go-Wa2} in the cases when $\Pi$ is of type 
$A_r$, $B_r$, $C_r$, $D_r$, or $E_6$
(note that when $\Pi$ is of type $B_r$ or $C_r$, there are two situations to be considered: when $\beta$ is a long root
and when $\beta$ is a short root). The results are nicely summarized in \cite[Theorem 3.3]{Go-Wa2}. 
The importance of Theorem \ref{thm:main} resides in that it automatically implies the complete integrability of the periodic quantum Toda lattice
associated to the extended Dynkin diagram $\Delta\cup\{-\beta\}$, as explained in \cite{Go-Wa1}.
More recently, in the case when $\beta$ is the highest root, the aforementioned integrability was established by 
Etingof in \cite{etingof}.  Thus the only remaining cases are when $\Pi$ is of type $F_4$ or $G_2$ and 
$\beta$ is short. These are the cases that are addressed in the main body of this paper.
For the sake of  completeness and clarity we have considered necessary to see exactly how Etingof's approach
can be adapted to prove Theorem \ref{thm:main} in the case when  $\beta$ is the highest root:
this is done in Appendix \ref{trei}.}

\begin{remark}
{\rm Theorem \ref{thm:main} was used in \cite{Ma1} and \cite{MM} in connection with the quantum cohomology ring of 
affine flag manifolds.}
\end{remark}

Here is an outline of the paper. { We first prove the uniqueness part in the theorem, see Section  
\ref{doi}. The main results are proved in  Sections \ref{patru} and \ref{cinci}, where we consider 
the cases when $\Pi$ is of one of the types $F_4$ or $G_2$, respectively, and $\beta$ is short.
Finally, Appendix \ref{trei} deals with Theorem \ref{thm:main} in the case when 
$\beta$ is the highest root, 
along the lines of Etingof's original proof from \cite{etingof}. }

\noindent{\bf Acknowledgements.} I would like to thank Leonardo Mihalcea for discussions about
the topics of this paper. I also thank the referee for valuable comments.

\section{Uniqueness of $\Omega_1,\ldots, \Omega_r$}\label{doi}
We continue the notation of the introduction. Since $\beta$ is a dominant root, in the expansion
\begin{equation}\label{domroot} \beta = m_1 \alpha_1 +\cdots + m_r \alpha_r\end{equation}
the coefficients $m_i$ are all strictly positive integers.
The following result is a direct consequence of \cite[Lemma 3.6]{Go-Wa1}.

\begin{proposition}\label{degree} {\rm (\cite{Go-Wa1})} If $\Gamma \in U(\mfb)_\eve$ commutes with $\Omega$ then $\Gamma$ is a 
multiple of $(X_0X_1^{m_1} \cdots X_r^{m_r})^2$. In particular, if $\Gamma \neq 0$, then 
the degree of $\Gamma$ is at least equal to $2 +2\sum_{i=1}^r m_i$.
\end{proposition}

The  table below contains the degrees of the fundamental $W$-invariant polynomials $u_1, \ldots, u_r$ 
and the sum of the coefficients $m_1, \ldots, m_r$ in equation (\ref{domroot}). The information is extracted from 
\cite[p.~477]{He} and  \cite[p.~372]{Go-Wa1}.   

\begin{tabular}{|c|l|l|l|}
\hline
Type & $\deg u_1, \ldots , \deg u_r$ & $\sum_{i=1}^r m_i$ ($\beta$ long) &  $\sum_{i=1}^r m_i$ ($\beta$ short)\\
\hline
$A_r$ ($r\ge 1$) & $2,3, \ldots , r+1$ & $r$ & $r$\\
$B_r$ ($r\ge 2$) & $2,4,6, \ldots , 2r$ & $2r-1$ & $r$\\
$C_r$ ($r\ge 2$) & $2,4,6,\ldots , 2r$ & $2r-1$ & $2r-2$\\
$D_r$ ($r\ge 3$) & $2,4,6,\ldots ,2r-2,r$ & $2r-3$ & $2r-3$\\
$E_6$ & $2,5,6,8,9,12$ & $11$ & $11$\\
$E_7$ & $2,6,8,10,12,14,18$ & $17$ & $17$\\
$E_8$ & $2,8,12,14,18,20,24,30$ & $29$ & $29$\\
$F_4$ & $2,6,8,12$ & $11$ & $8$\\
$G_2$ & $2,6$  & $5$ & $3$\\
\hline
\end{tabular}

The table  shows that $\deg u_1,\ldots,\deg u_r$ are always at most equal to $2\sum_{i=1}^rm_i$.
Thus Proposition \ref{degree} immediately  implies as follows. 

\begin{corollary} For fixed $i\in \{1, \ldots, r\}$, there exists at most one 
$\Omega_i \in U(\mfb)_{\eve}$ such that $[\Omega_i,\Omega]=0$, $\deg \Omega_i =\deg u_i$, and 
the symbol of $\Omega_i$ is equal to $u_i$.
\end{corollary}

   \section{The case when $\Pi$ is  of type $F_4$  and $\beta$ is short}\label{patru}
   
   We will prove Theorem \ref{thm:main} in the case mentioned in the title. 
     The proof's main tool is the folding of the extended (untwisted) Dynkin diagram $E_7^{(1)}$
     onto the extended twisted Dynkin diagram  $E_6^{(2)}$.  Figure \ref{E6}
     below is taken from  \cite[Section 3.3]{BST} and describes the folding map.
     \begin{figure}[h]
\begin{center}
\epsfig{figure=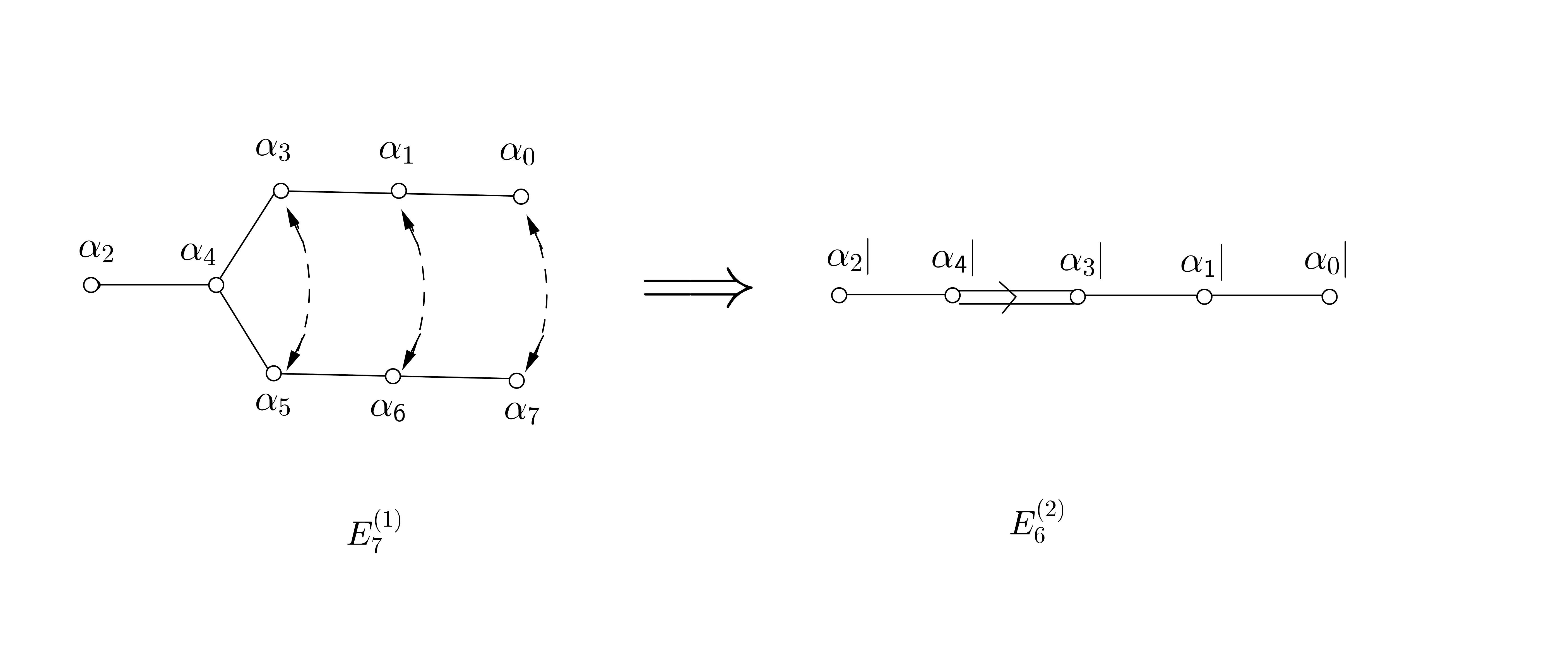,width=5in}
\end{center}
\begin{center}
\parbox{12in}{
\caption[]{The folding $E_7^{(1)}\to E_6^{(2)}.$ }\label{E6} }
\end{center}
\end{figure}   
 
 \noindent More specifically, let $\Pi \subset \mfa^*$ be the root system of type $E_7$ and pick  simple roots 
 $\{\alpha_1,\ldots,\alpha_7\}$ fitting the diagram above. As usual, $\alpha_0=-\theta$, where  
 $$\theta=2\alpha_1+2\alpha_2+3\alpha_3+4\alpha_4+3\alpha_5+2\alpha_6+\alpha_7$$ is  the corresponding highest root.
 Consider the automorphism $\tau$ of $\mfa^*$ which is determined by
       $$\tau(\alpha_1)=\alpha_6, \ \tau(\alpha_2) = \alpha_2, \ \tau(\alpha_3)=\alpha_5, \ 
       \tau(\alpha_4)=\alpha_4, \ \tau(\alpha_5) = \alpha_3, \ \tau(\alpha_6) = \alpha_1, \ \tau(\alpha_7) = \alpha_0.$$ 
       It follows that $\tau(\alpha_0)=\alpha_7$, hence $\tau$ is an  automorphism of order two  of $\mfa^*$ which permutes the elements of
       $\Pi$. Since $\mfa$ and $\mfa^*$ are canonically isomorphic, $\tau$ is also an automorphism 
       of $\mfa$. Let $\mfa'$ be the fixed point set of $\tau$.  
     Consider the restriction map $\mfa^*\to(\mfa')^*$ and denote the image of an arbitrary $\lambda\in \mfa^*$ by $\lambda|$. 
     The elements of $\Pi$ are mapped to a root system of type $F_4$, which we denote by $\Pi|$. 
     A simple root system for $\Pi|$ is $\{\alpha_1|$, $\alpha_2|$, $\alpha_3|$, $\alpha_4|\}$. Moreover, 
     $$   \theta| = 4\alpha_1| + 2\alpha_2| + 6\alpha_3| + 4\alpha_4|-\theta|$$
     thus
     $$\theta|= 2\alpha_1|+\alpha_2|+ 3\alpha_3| +2\alpha_4|,$$
     which is the short dominant root of $\Pi|$ (see \cite[p.~372]{Go-Wa1}). 
     
   By the definition in Section \ref{first}, the  ``$ax+b$" algebra associated to the extended Dynkin diagram $E_7^{(1)}$ is  $\mfb = \mfa\oplus {\rm Span}_\bR\{X_0, \ldots, X_7\}$, equipped with  the Lie bracket which is
   identically zero on both $\mfa$ and  ${\rm Span}_\bR\{X_0, \ldots, X_7\}$ and such that
   $$[H, X_i]=\alpha_i (H) X_i, \ {\rm for \ all \ } 0\le i \le 7 \ {\rm and  \ all \ } H \in \mfa.$$  
   We now extend $\tau:\mfa\to \mfa$ to the linear automorphism $\tau : \mfb \to \mfb$ such that:
   \begin{align*}{}&\tau(X_0) = X_7, \ \tau(X_1) = X_6, \ \tau (X_2) = X_2, \ \tau(X_3) = X_5, \\{}& \tau (X_4)=X_4, \ \tau(X_5) = X_3, \  \tau(X_6)=X_1, \
   \tau(X_7)=X_0.\end{align*}
   
   \begin{lemma}\label{autom} The map $\tau : \mfb \to \mfb$ is a Lie algebra automorphism.
   \end{lemma}
   
   \begin{proof} One needs to check that $\tau([H, X_i]) =[\tau(H), \tau(X_i)]$, for all $H\in \mfa$ and all $0\le i \le 7.$
    We will confine ourselves to the case when $i=0$. We have $$\tau([H, X_0]) = \tau(\alpha_0(H)X_0)=\alpha_0(H)X_7,$$
    whereas
    $$[\tau(H),\tau(X_0)]= [\tau(H), X_7]=\alpha_7(\tau(H))X_7 =\alpha_0(H)X_7.$$
       \end{proof}

The eigenspace decomposition of $\tau$ is 
$\mfb = \mfb'\oplus \mfb''$, where $\mfb':=\{Z\in \mfb \mid \tau(Z)=Z\}$ and $\mfb''=\{Z\in \mfb \mid \tau(Z)=-Z\}$.
Note that $X_0+X_7, X_1+X_6, X_3+X_5, X_2, X_4\in \mfb'$ and 
 for any $H'\in \mfa'$ one has
 \begin{align} {}& [H', X_0+X_7]=\alpha_0(H')(X_0+X_7)\nonumber\\
 {}& [H', X_1+X_6]=\alpha_1(H')(X_1+X_6)\nonumber\\
 {}& [H', X_2]=\alpha_2(H')X_2\label{hx}\\
 {}& [H', X_3+X_5]=\alpha_3(H')(X_3+X_5)\nonumber\\
 {}& [H', X_4]=\alpha_4(H')X_4.\nonumber
 \end{align}
 That is, $\mfb'=\mfa'\oplus \mfu'$, where $\mfu'$ has a basis consisting of
 $$X'_0:=\frac{1}{\sqrt{2}}(X_0+X_7), X'_1:=\frac{1}{\sqrt{2}}(X_1+X_6), X'_2:=X_2, X_3':=\frac{1}{\sqrt{2}}(X_3+X_5), X_4':=X_4.$$
 We have inserted  the factors of $1/\sqrt{2}$ for later use, see equation (\ref{concl}) below.  
 Equations (\ref{hx}) are telling us that $\mfb'$ is a Lie subalgebra of $\mfb$ isomorphic to the ``$ax+b$" algebra that corresponds to the Dynkin diagram $E_6^{(2)}$ (that is, $F_4$ extended with the short dominant root).

Denote by $\langle U(\mfb)\mfb''\rangle$ the linear span of $U(\mfb)\mfb''$.

\begin{lemma}\label{uofb}
 One has \begin{equation}\label{decomp}U(\mfb) = U(\mfb')\oplus \langle U(\mfb)\mfb''\rangle.\end{equation}
 \end{lemma}
 
 \begin{proof}  Pick bases $Y_1,\ldots, Y_9$ of $\mfb'$ and $Z_1,\ldots, Z_6$ of $\mfb''$. 
The corresponding PBW basis of $U(\mfb)$ consists of $Y^IZ^J$, where $I=(i_1,\ldots, i_9)$ and  $J=(j_1,\ldots, j_6)$ are vectors
with components in $\bZ_{\ge 0}$. The PBW basis of $U(\mfb')$ is $\{Y^I \}$,
hence
$U(\mfb) = U(\mfb')+ \langle U(\mfb)\mfb''\rangle$. We now show that the latter sum is direct. 
To this end, note that if $I$ and $J$ are  as before and $k\in\{1,\ldots, 6\}$, then 
$Y^IZ^JZ_k$ is equal to $Y^IZ_1^{j_1} \cdots Z_k^{j_k+1} \cdots Z_6^{j_6}$ plus terms whose degree is strictly smaller
than the degree of $Y^IZ^JZ_k$. Hence if  $\Gamma \in \langle U(\mfb)\mfb''\rangle$ has degree $m$ then
in its PBW expansion the degree $m$ monomials are of the form  $Y^IZ^J$, where at least one of $j_1,\ldots, j_6$ is non-zero; thus $\Gamma$ cannot be in $U(\mfb')$. 
\end{proof}

Denote by $\nu : U(\mfb)\to U(\mfb')$ the projection onto the first component of the decomposition
(\ref{decomp}). We will also need the eigenspace decomposition of $\tau|_{\mfa}$, which is $\mfa=\mfa'\oplus \mfa''$;
 note that $\nu$ maps $\mfa$ onto $\mfa'$ as the orthogonal projection map.

 \begin{remark} {\rm The map $\nu : U(\mfb)\to U(\mfb')$ is not a homomorphism of (associative) algebras.
 More precisely, its restriction $\mfb \to \mfb'$ is not a homomorphism of Lie algebras.  
For $\nu(X_0)=\nu(X_7)=\frac{1}{2}(X_0+X_7)$ and if $H\in \mfa$ is such that
$\alpha_0(H)=0$ and $\alpha_7(H)=1$ then 
$$[H, X_0]=0 \ {\rm and} \ [H, X_7]=X_7.$$
If $\nu$ were a homomorphism, then we would have simultaneously
$$[\nu(H) , X_0+X_7]=0 \ {\rm and} \ [\nu(H), X_0+X_7] = X_0+X_7,$$
which is impossible.}   
\end{remark}      
 
\begin{lemma}\label{even} The map $\nu$ transforms elements of $U(\mfb)_{\rm ev}$ into elements of
$U(\mfb')_{\rm ev}$. \end{lemma} 

\begin{proof} Start with the observation that an element of $U(\mfb)$ is in $U(\mfb)_{\rm ev}$ if and only
if it is fixed by each of the eight algebra automorphisms  $\rho_0, \ldots, \rho_7: U(\mfb)\to U(\mfb)$,
where $\rho_i$ is the identity map on $\mfa$ and 
$$\rho_i(X_i) :=-X_i, \ \rho_i( X_j) :=X_j, \ {\rm for} \ j \neq i.$$
Observe that $\rho_2\circ \tau =\tau\circ \rho_2$, hence $\rho_2$ preserves the
 decomposition (\ref{decomp}). The same holds for $\rho_4$.  
 Consider now  $\rho_{07}:=\rho_0\circ \rho_7 =\rho_7\circ \rho_0$, which is the identity on $\mfa$ and:
 $$\rho_{07}(X_0)=-X_0, \ \rho_{07}(X_7)=-X_7, \ \rho_{07}(X_j) = X_j,  \ {\rm for} \ j \neq 0,7.$$  
 One has $\rho_{07}\circ \tau = \tau \circ \rho_{07}$, hence $\rho_{07}(U(\mfb')) = U(\mfb')$
 and $\rho_{07}(\langle U(\mfb)\mfb''\rangle) =\langle U(\mfb)\mfb''\rangle$.  
 In the same way one defines $\rho_{16}$ and $\rho_{35}$ and note that they both preserve the
 decomposition (\ref{decomp}). The result stated in the lemma follows from the fact that
 if $\Gamma\in U(\mfb)_{\rm ev}$ then $\rho_i(\Gamma)=\Gamma$, for all $0\le i \le 7$, hence
 $$\rho_2(\nu(\Gamma))=\nu(\Gamma), \ \rho_4(\nu(\Gamma)) = \nu(\Gamma), \ \rho_{07}(\nu(\gamma)) = \nu(\Gamma),\
 \rho_{16}(\nu(\Gamma)) =\nu(\Gamma), \ {\rm and} \ \rho_{35}(\nu(\Gamma)) =\nu(\Gamma).$$
 This implies that $\nu(\Gamma) \in U(\mfb')_{\rm ev}$.  
\end{proof}

Recall now that the degrees of the fundamental homogeneous generators for
    ${\rm Sym}(\mfa)^{W_{E_7}}$ and ${\rm Sym}(\mfa')^{W_{F_4}}$ are $2, 6, 8, 10, 12, 14, 18$ and $2, 6, 8, 12$ respectively,
    see the table in Section \ref{doi}. 
    
Here is the main result of this section:

\begin{proposition}\label{u1u2} (a) There exists a system $u_1,\ldots, u_7$ of fundamental generators of ${\rm Sym}(\mfa)^{W_{E_7}}$ 
of degrees $2, 6, 8, 10, 12, 14, 18$ respectively, such that $\nu(u_1), \nu(u_2), \nu(u_3),$ and  $\nu(u_5)$
are fundamental generators of ${\rm Sym}(\mfa')^{W_{F_4}}$.

(b) Let $\Omega\in U_2(\mfb)$ be the Laplacian and  $\Omega_1,\ldots, \Omega_7$  the elements of $U(\mfb)$ associated to $u_1,\ldots, u_7$ by Proposition
\ref{long}. Then $\Omega':=\nu(\Omega)$ is a multiple of the Laplacian of $\mfb'$.
Furthermore, 
$\Omega_1':=\nu(\Omega_1), \Omega_2':=\nu(\Omega_2), \Omega_3':=\nu(\Omega_3),$ and $\Omega_5':=\nu(\Omega_5)$ have symbols
equal to $ \nu(u_1), \nu(u_2), \nu(u_3),$ and $\nu(u_5)$, respectively. They all lie in $U(\mfb')_{\rm ev}$
 and commute with each other.
\end{proposition}

 In what follows we will prove Proposition \ref{u1u2}. We start with a concrete presentation of the roots of
 type $E_7$.
 Take $\mfa=\bR^7$; the positive roots are the following linear functions on $\mfa$ (as usual, $x_i$ are the coordinate functions):
     \begin{align*} 2x_i, \quad 1\le i \le 7;\quad & x_1\pm x_2\pm x_3\pm x_4;\\
     x_1\pm x_2\pm x_5\pm x_6;\quad & x_1\pm x_3\pm x_5\pm x_7; \\
     x_1\pm x_4\pm x_6\pm x_7;\quad &  x_2\pm x_3\pm x_6\pm x_7;\\
      x_2\pm x_4\pm x_5\pm x_7;\quad & x_3\pm x_4\pm x_5\pm x_6.
      \end{align*}
     A simple root system that fits the Dynkin diagram in Figure 1  is:
     \begin{align*}{}& \alpha_1=x_2-x_3+x_6-x_7, \ \alpha_2=x_1-x_2-x_3-x_4, \ \alpha_3=x_3-x_4+x_5-x_6,  \\
     {}&  \alpha_4=2x_4, \ \alpha_5=x_3-x_4-x_5+x_6, \ \alpha_6=x_2-x_3-x_6+x_7, \ \alpha_7=-x_1-x_2-x_5-x_6.
     \end{align*}
     The resulting highest root is $\theta= x_1+x_2-x_5-x_6$. 
     The involution $\tau$ is the linear transformation
     $$e_i \mapsto e_i, \  1\le i \le 4; \ e_j\mapsto - e_j, \ 5\le j \le 7,$$
     where $e_1,\ldots, e_7$ is the canonical basis of $\bR^7$;  
     thus $\mfa'=\bR^4$. The restricted root system is of type $F_4$ and admits the following simple roots:
     $$\alpha_2|=x_1-x_2-x_3-x_4, \  \alpha_4|= 2x_4,  \ \alpha_3|= x_3-x_4, \ \alpha_1|= x_2-x_3.$$
     
   According to \cite[p.~1096]{Me}, the following is a system of fundamental $W_{E_7}$-invariant elements of ${\rm Sym}(\mfa)$:
\begin{align}\label{un}{}& v_k:=\sum_{\pm}\left(e_1\pm e_2\pm e_7\right)^{k}+\sum_{\pm}\left(e_1\pm e_3\pm e_6\right)^{k} +\sum_{\pm}\left(e_1\pm e_4\pm e_5\right)^k +\sum_{\pm}\left( e_2\pm e_3\pm e_5\right)^k\\{}& +\sum_{\pm}\left(e_2\pm e_4\pm e_6\right)^k+\sum_{\pm}\left( e_3\pm e_4\pm e_7\right)^k+\sum_{\pm}\left( e_5\pm e_6\pm e_7\right)^k,\nonumber
  \end{align}
  where $k=2, 6, 8, 10, 12, 14, 18$ (here $\Sigma_{\pm}$ denotes the sum over all four possible choices of the two signs).  
   For $k=2, 6, 8, 12$, the projection $\nu:{\rm Sym}(\mfa) \to {\rm Sym}(\mfa')$ maps the expressions above to 
  \begin{align}\label{do}
\nu(v_k)=2\sum_{1\le i<j\le 4}(e_i-e_j)^{k}+(e_i+e_j)^{k}, \quad k=2,6,8,12,
  \end{align}  
  which are a system of fundamental $W_{F_4}$-invariant elements of ${\rm Sym}(\mfa')$, see \cite[Section 2.4]{Me}. 
  In this way we have proved point (a) of Proposition \ref{u1u2}. 
  
  We now prove point (b).  
  One has
$$\Omega=u_1+X_0^2 +\cdots + X_7^2, \quad {\rm where} \ u_1=v_2.$$ 
Observe that
\begin{align}
{} & X_0^2+ X_7^2= \frac{1}{2}(X_0+X_7)^2+\frac{1}{2}(X_0-X_7)^2=(X_0')^2+\frac{1}{2}(X_0-X_7)^2\label{erste}\\
{} & X_1^2+ X_6^2= \frac{1}{2}(X_1+X_6)^2+\frac{1}{2}(X_1-X_6)^2=(X_1')^2+\frac{1}{2}(X_1-X_6)^2\\
{} & X_3^2+ X_5^2= \frac{1}{2}(X_3+X_5)^2+\frac{1}{2}(X_3-X_5)^2=(X_3')^2+\frac{1}{2}(X_3-X_5)^2.\label{dritte}
\end{align}
Since $X_0-X_7, X_1-X_6$, and $X_3-X_5$ are in $\mfb''$, we deduce that  
\begin{equation}\label{concl}\nu(\Omega_1) =\nu(u_1) + (X_0')^2 + (X_1')^2 + (X_2')^2 + (X_3')^2+ (X_4')^2.\end{equation}
 By equation (\ref{laplace}), this  is a scalar multiple of the Laplacian of $\mfb'$
 (the reason is  that $\nu(u_1)$ is a scalar multiple of $\sum_{i,j=1}^4 \langle \alpha_i|,\alpha_j|\rangle H'_i H'_j$,
 where $H'_1,\ldots, H'_4\in \mfa'$ with $\alpha_i(H'_j) =\delta_{ij}$, $1\le i,j \le 4$;
 one may need to rescale the metric on $\mfu'$ such that $\{X_0', X_1', X_2'$, $X_3'$, $X_4'\}$  is an orthonormal basis). 
  
 Let $\tau : U(\mfb) \to U(\mfb)$ be the  algebra automorphism induced  by Lemma \ref{autom}.

\begin{lemma} \label{abp}
(a) If $\Gamma\in U(\mfb')$ then $\tau(\Gamma)=\Gamma$.

(b) If $\Gamma\in U(\mfb)$ has the property that $\tau(\Gamma)=-\Gamma$, then
$\Gamma\in \langle U(\mfb)\mfb''\rangle$.
\end{lemma}  
  
\begin{proof} Only (b) needs to be justified.  By Lemma \ref{uofb}, one can write $\Gamma=\Gamma'+\Gamma''$, where $\Gamma'\in U(\mfb')$ and
$\Gamma''\in \langle U(\mfb)\mfb''\rangle$. This implies that $\tau(\Gamma)=\Gamma'+\tau(\Gamma'')$.
Since $\tau(\Gamma'')\in \langle U(\mfb)\mfb''\rangle$, one has $\Gamma'=0$.
\end{proof}

  It now comes the last step in the proof of Proposition \ref{u1u2}.
  
  \begin{lemma} (a) For any $i\in \{1,2,3,5\}$, the symbol of $\nu(\Omega_i)$ in $U(\mfb')$ is $\nu(u_i)$.
  
 (b) For any two $i,j\in \{1, 2,3,5\}$ one has $[\nu(\Omega_i), \nu(\Omega_j)]=0$.
  \end{lemma}  
  
  \begin{proof} 
(a) Let $\mu': U(\mfb')\to U(\mfa')$ be the symbol map. It is sufficient to show that the following diagram is commutative.

\[
\xymatrix{
U(\mfb) \ar[r]^{\mu} \ar[d]^{\nu} & U(\mfa) \ar[d]^{\nu} \\
U(\mfb') \ar[r]^{\mu'} & U(\mfa')
}
\]

Take $\Gamma \in U(\mfb)$ and decompose it as $\Gamma = \Gamma' +\sum_{k=1}^p a_k B_k$ where 
$\Gamma' \in U(\mfb')$, $a_k \in U(\mfb)$, and $B_k \in \mfb''$. 
Both $\nu : \mfa\to \mfa'$ and $\mu : \mfb \to \mfa$ are Lie algebra homomorphisms,
hence $\nu \circ \mu (\sum_{k=1}^p a_k B_k) = \sum_{k=1}^p \nu(\mu(a_k)) \nu(\mu(B_k)) =0$,
since $\nu(\mu(B_k))=0$ (recall that
$\mfb''=\mfa''\oplus {\rm Span}\{X_0-X_7, X_1-X_6, X_3-X_5\}$).   It remains to show that $\nu\circ \mu (\Gamma') = \mu' (\Gamma')$. 
But this follows readily from the fact that $\nu \circ \mu|_{\mfb'} : \mfb'\to \mfa'$ is equal to
$\mu'$, the map being a Lie algebra homomorphism. 

(b) Decompose
\begin{align*}
{}& \Omega_i = \nu(\Omega_i)+\Omega_i''\\
{}& \Omega_j = \nu(\Omega_j)+\Omega_j'',
\end{align*}
where  $\Omega_i'', \Omega_j'' \in \langle U(\mfb)\mfb''\rangle$. From $\Omega_i \Omega_j = \Omega_j \Omega_i$ 
one obtains:
\begin{align}\label{nuomega}{}& \ \ \ \ \nu(\Omega_i)\nu(\Omega_j) + \nu(\Omega_i)\Omega_j''+\Omega_i''\nu(\Omega_j) +\Omega_i''\Omega_j''\\
{}& =\nu(\Omega_j)\nu(\Omega_i) + \nu(\Omega_j)\Omega_i''+\Omega_j''\nu(\Omega_i) +\Omega_j''\Omega_i''.\nonumber\end{align}  
We claim that except  $\nu(\Omega_i)\nu(\Omega_j)$ and $\nu(\Omega_j)\nu(\Omega_i)$, all terms in both sides 
of the equation above are in $\langle U(\mfb)\mfb''\rangle$. The claim is clearly true for
$\nu(\Omega_i)\Omega_j'', \nu(\Omega_j)\Omega_i'', \Omega_i''\Omega_j''$, and $\Omega_j''\Omega_i''$. 
To show that $\Omega_j''\nu(\Omega_i)$ is in $\langle U(\mfb)\mfb''\rangle$, we note that
$\Omega_j'' = \sum_{k=1}^p a_k B_k$ for some  $a_k \in U(\mfb)$ and $B_k \in \mfb''$.
Hence $\Omega_j'' \nu(\Omega_i) = \sum_{k=1}^p a_k B_k \nu(\Omega_i)$.  
But by Lemma \ref{abp} (a), for any $1\le k \le p$ we have $\tau(B_k\nu(\Omega_i))=\tau(B_k)\tau(\nu(\Omega_i)) = -B_k \nu(\Omega_i)$;
 hence, by Lemma \ref{abp} (b), the product $B_k \nu(\Omega_i)$ is in $\langle U(\mfb)\mfb''\rangle$. 
By multiplying each such product from the left by $a_k$ and adding up all these expressions ($1\le k \le p$),
the result will be in  $\langle U(\mfb)\mfb''\rangle$. 
  \end{proof}

 As a consequence of Proposition \ref{u1u2}, we have:
 
 \begin{corollary}\label{cor:f4}
{ The conclusion of Theorem \ref{thm:main} is} true in the case when $\Pi$ is of type $F_4$ and $\beta$ is the short dominant root. 
\end{corollary}

 \begin{proof} Proposition \ref{u1u2} is telling us that { the conclusion of Theorem \ref{thm:main} is} true for a special choice of the
 fundamental $W_{F_4}$-invariant homogeneous polynomials, namely $\nu(u_1), \nu(u_2), \nu(u_3)$, and $\nu(u_5)$. 
 Let now $u_1', u_2', u_3', u_4'\in S(\mfa')^{W_{F_4}}$ be an arbitrary system of such polynomials. 
 For each $j=1,2,3,4$, one can write $u_j'$ as $f_j(\nu(u_1), \nu(u_2), \nu(u_3),\nu(u_5))$,
 where $f_j$ is a polynomial in four variables. Then $\Omega'_j:=f_j(\nu(\Omega_1), \nu(\Omega_2), \nu(\Omega_3),
 \nu(\Omega_5))$, where $1\le j \le 4$,  satisfy the conditions in Proposition \ref{u1u2} relative to
 $u_1', u_2', u_3', u_4'$.  
 \end{proof}  
  
   \section{The case when $\Pi$ is  of type  $G_2$ and $\beta$ is short}\label{cinci}
     
   In the case mentioned in the title, we will use the same approach as in the previous section, this time for the folding of the  extended Dynkin diagram $E_6^{(1)}$
     onto the extended twisted Dynkin diagram  $D_4^{(3)}$.
     This is described in Figure \ref{F4} below, see \cite[Section 3.4]{BST}.
    \begin{figure}[h]
\begin{center}
\epsfig{figure=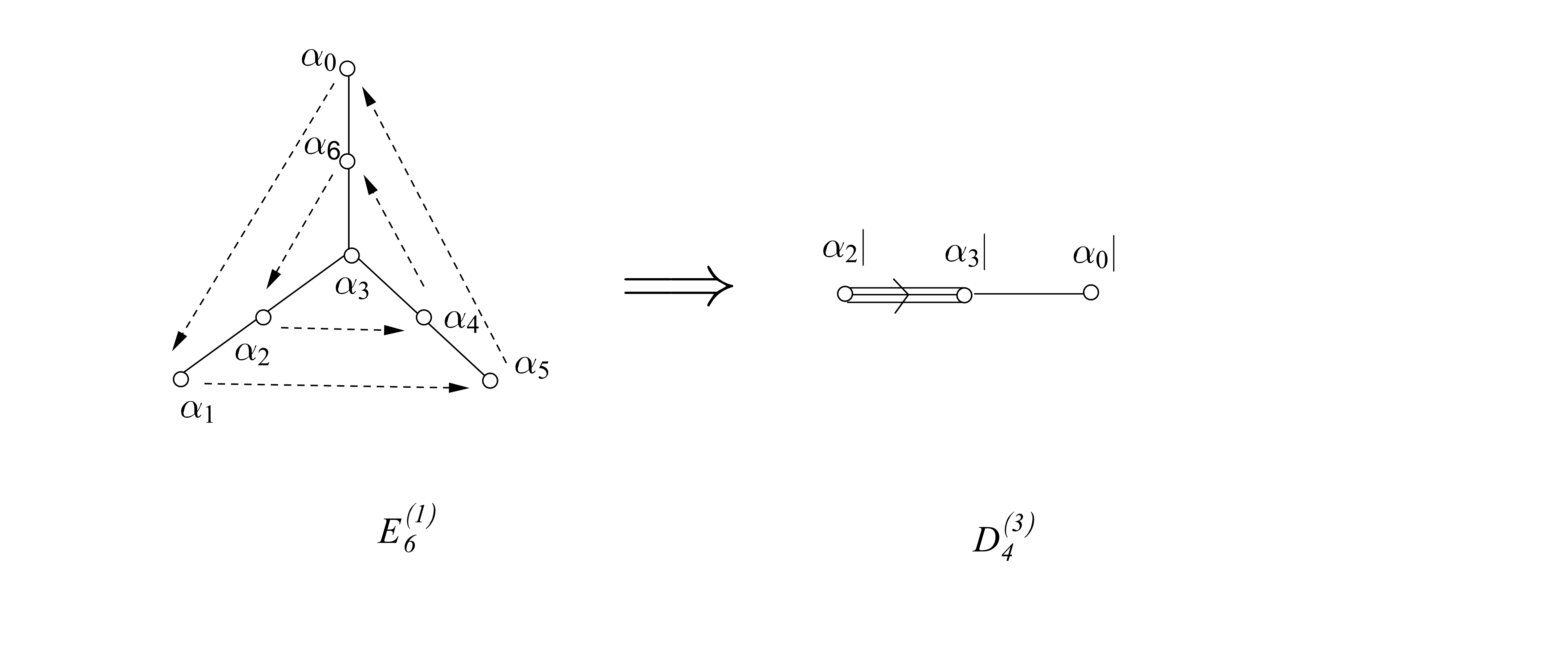,width=5in}
\end{center}
\begin{center}
\parbox{12in}{
\caption[]{The folding $E_6^{(1)}\to D_4^{(3)}.$ }\label{F4} }
\end{center}
\end{figure} 
   \newline \noindent Here $\{\alpha_1, \ldots, \alpha_6\}\subset \mfa^*$ are simple roots of the root system $\Pi$ of type $E_6$,
 and $\alpha_0=-\theta$ is the negative of the highest root. One has
 $$\theta=\alpha_1+2\alpha_2+3\alpha_3+2\alpha_4+\alpha_5+2\alpha_6.$$
 Let $\tau$ be the automorphism of $\mfa^*$ determined by
 $$\tau(\alpha_1)= \alpha_5, \ \tau(\alpha_2) = \alpha_4, \ \tau(\alpha_3) = \alpha_3, \ \tau(\alpha_4) = \alpha_6, \
 \tau(\alpha_5) = \alpha_0, \ \tau(\alpha_6) = \alpha_2.$$
 One deduces that $\tau(\alpha_0) = \alpha_1$, hence $\tau$ is a diagram automorphism of order three, that is, $\tau^3={\rm id}$. 
 The canonical isomorphism $\mfa^*\simeq \mfa$ induces an automorphism $\tau$ of $\mfa$.  
 Let $\mfa'\subset \mfa$ be the fixed point set of $\tau$. For $\lambda \in \mfa^*$, we denote by $\lambda|$ the restriction to $\mfa'$.
  Then  $\Pi|:=\{\alpha| \ ;  \ \alpha\in \Pi\}$  is a root system of type $G_2$. It admits $\{\alpha_2|, \alpha_3|\}$ as simple roots.
One has $\theta|= 2\alpha_2| + \alpha_3|$, which is the short dominant root of $\Pi|$ (see \cite[p.~372]{Go-Wa1}).

 The  ``$ax+b$" algebra of the extended Dynkin diagram $E_6^{(1)}$ is
 $\mfb = \mfa\oplus {\rm Span}_\bR\{X_0, \ldots, X_6\}$ where
   \begin{align*}{}& [H, X_i]=\alpha_i (H) X_i, \ {\rm for \ all \ } 0\le i \le 6 \ {\rm and  \ all \ } H \in \mfa.\end{align*}
Let $\tau: \mfb\to \mfb$ be the linear automorphism which extends the previous $\tau$  such that
 \begin{align*}{}& \tau(X_0) = X_1, \ \tau(X_1)= X_5, \ \tau(X_2) = X_4, \ \tau(X_3) = X_3, \\ {}& \tau(X_4) = X_6, \
 \tau(X_5) = X_0, \ \tau(X_6) = X_2.\end{align*}
 One shows that $\tau$ is a Lie algebra automorphism, by proceeding in the same way as in Lemma \ref{autom}. 
 
 Let $\mfb'\subset \mfb$ be the fixed point set of $\tau$. Observe that $X_0+X_1+X_5, X_2+X_4+X_6,$ and $X_3$ are in $\mfb'$ 
 and that for any $H'\in \mfa'$, one has
 \begin{align} {}& [H', X_0+X_1+X_5]=\alpha_0(H')(X_0+X_1+X_5),\nonumber\\
 {}& [H', X_2+X_4+X_6]=\alpha_2(H')(X_2+X_4+X_6),\label{hxy}\\
 {}& [H', X_3]=\alpha_3(H')X_3.\nonumber
 \end{align}
 That is, $\mfb'=\mfa'\oplus \mfu'$, where $\mfu'$ has a basis consisting of
 $$X'_0:=\frac{1}{\sqrt{3}}(X_0+X_1+X_5), X'_2:=\frac{1}{\sqrt{3}}(X_2+X_4+X_6), X_3':=X_3.$$
 By (\ref{hxy})  $\mfb'$ is  isomorphic to the ``$ax+b$" algebra that corresponds to the Dynkin diagram $D_4^{(3)}$. 
 
 Let $\mfb_\bC$ and $\mfb'_\bC$ be the complexifications of $\mfb$ and $\mfb'$. Extend $\tau$ by $\bC$-linearity to an automorphism of $\mfb_\bC$. 
 Set $\ve := \cos (\pi/3) + i \sin (\pi/3)$. The eigenspace decomposition of $\tau$ is $\mfb_\bC=\mfb_\bC'\oplus \mfb''$, where
 $$\mfb''=\{Z\in \mfb_\bC \mid \tau(Z)=\ve Z\}\oplus \{Z\in \mfb_\bC \mid \tau(Z)=\ve^2Z\}.$$ 
 The following result can be proved with the same method as Lemma \ref{uofb}. 
 
 \begin{lemma}
One has \begin{equation}\label{ubprim} U(\mfb_\bC) = U(\mfb_\bC')\oplus \langle U(\mfb_\bC)\mfb''\rangle,\end{equation}
where $\langle U(\mfb_\bC)\mfb''\rangle$ is the $\bC$-linear span of $U(\mfb_\bC)\mfb''$.
 \end{lemma}
 
We denote by $\nu:U(\mfb_\bC) \to  U(\mfb_\bC')$ the projection onto the first component of
the decomposition (\ref{ubprim}). Observe that $\nu$ maps $\mfa$ to $\mfa'$ as the orthogonal projection map. 
 
 In the same way as Lemma \ref{even}, we have:
\begin{lemma} The map $\nu$ transforms elements of $U(\mfb_\bC)_{\rm ev}$ into elements of
$U(\mfb'_\bC)_{\rm ev}$. \end{lemma}  
 
    From the table in Section 2,  the degrees of the fundamental homogeneous generators for
    ${\rm Sym}(\mfa)^{W_{E_6}}$  are $2, 5, 6, 8, 9,$ and $12$, whereas  for ${\rm Sym}(\mfa')^{W_{G_2}}$ they are $2$
    and  $6.$ We will prove as follows:
    
    \begin{proposition}\label{u1u3} (a) There exists a system $u_1,\ldots, u_6$ of fundamental generators of ${\rm Sym}(\mfa)^{W_{E_6}}$ 
of degrees $2, 5, 6, 8,9,$ and $12$ respectively, such that $\nu(u_1)$ and  $\nu(u_3)$
are fundamental generators of ${\rm Sym}(\mfa')^{W_{G_2}}$.

(b) Let $\Omega\in U_2(\mfb)$ be the Laplacian and  $\Omega_1,\ldots, \Omega_6$  the elements of $U(\mfb)$ associated to $u_1,\ldots, u_6$ by Proposition
\ref{long}. Then $\Omega':=\nu(\Omega)$ is a multiple of the Laplacian of $\mfb'$.
Furthermore, 
$\Omega_3':=\nu(\Omega_3)$  belongs to $U(\mfb_\bC')_{\rm ev}$, has its symbol
equal to $\nu(u_3)$, 
 and commutes with $\Omega'$.
\end{proposition}

   To prove the proposition, we need the following  presentation of the root system of type $E_6$, see \cite[p.~1095]{Me}.
    One takes $\mfa = \bR^6$.  
     The positive roots are:
     
     \begin{align*} 2x_i, \quad 1\le i \le 4;\quad & x_1\pm x_2\pm x_3\pm x_4;\\
     x_1\pm x_2 \pm\sqrt{2}x_5;\quad & x_3\pm x_4 \pm\sqrt{2}x_5; \\
     x_1\pm x_3\pm \frac{1}{\sqrt{2}}(x_5-\sqrt{3}x_6);\quad &  x_2\pm x_4\pm \frac{1}{\sqrt{2}}(x_5-\sqrt{3}x_6);\\
        x_1\pm x_4\pm \frac{1}{\sqrt{2}}(x_5+\sqrt{3}x_6);\quad & x_2\pm x_3\pm \frac{1}{\sqrt{2}}(x_5+\sqrt{3}x_6).
      \end{align*}
      A simple root system is:
      \begin{align*}{}&\alpha_1=x_2- x_3- \frac{1}{\sqrt{2}}(x_5+\sqrt{3}x_6),\quad \alpha_2=x_3- x_4 +\sqrt{2}x_5, \quad \alpha_3=2x_4, \\{}& \alpha_4=x_3- x_4 -\sqrt{2}x_5,\quad
      \alpha_5=x_2- x_3+ \frac{1}{\sqrt{2}}(x_5+\sqrt{3}x_6),\quad \alpha_6=x_1-x_2-x_3-x_4.\end{align*}
 The highest root is $$\theta= \alpha_1+ 2\alpha_2+ 3\alpha_3+ 2\alpha_4 + \alpha_5 + 2\alpha_6 =2x_1.$$
 Thus $\alpha_0=-2x_1$. 
 
The subspace $\mfa'$ of $\mfa$ is described by  the equations
 $\alpha_2=\alpha_4=\alpha_6$ and $\alpha_0=\alpha_1=\alpha_5$; they are equivalent to:
 $$x_1=-x_2=x_3, \quad x_5=x_6=0.$$
 A simple root system for $\Pi|$ consists of the following two linear functions on $\mfa'$: 
 $$\alpha_2| = x_1-x_4 \ {\rm and} \ \alpha_3|=2x_4.$$
The short positive roots are   
$x_1-x_4,-2x_1,$ and $-x_1-x_4.$

 The following is a system of fundamental $W_{E_6}$-invariant elements of 
 ${\rm Sym}(\mfa^*)$ (see \cite[p.~1096]{Me}):
\begin{align*}&v_k:=\left(2\sqrt{\frac{2}{3}}x_6\right)^{k}+\left(\sqrt{\frac{2}{3}}(\sqrt{3}x_5-x_6)\right)^{k}+\left(\sqrt{\frac{2}{3}}(-\sqrt{3}x_5-x_6)\right)^{k}\\{}&
+\sum_{\pm}\left( \pm x_3\pm x_4 -\sqrt{\frac{2}{3}}x_6\right)^k+\sum_{\pm}\left( \pm x_1\pm x_2 -\sqrt{\frac{2}{3}}x_6\right)^k \\
{}& +\sum_{\pm}\left(\pm x_2\pm x_4+\frac{1}{\sqrt{6}}(\sqrt{3} x_5+x_6)\right)^k +\sum_{\pm}\left(\pm x_1\pm x_3+\frac{1}{\sqrt{6}}(\sqrt{3} x_5+x_6)\right)^k\\{}& +\sum_{\pm}\left(\pm x_2\pm x_3-\frac{1}{\sqrt{6}}(\sqrt{3} x_5-x_6)\right)^k+\sum_{\pm}\left(\pm x_1\pm x_4-\frac{1}{\sqrt{6}}(\sqrt{3} x_5+x_6)\right)^k,
  \end{align*}
 where  $k=2,5,6,8,9,12$ (here $\Sigma_{\pm}$ denotes the sum over all four possible choices of the two signs).
 We are particularly interested in the cases $k=2$ and $k=6$. 
 The images of  $v_2$ and $v_6$ by the restriction map $\nu : {\rm Sym}(\mfa^*)\to {\rm Sym}((\mfa')^*)$  are
\begin{align*} {}& \nu(v_2) = 6\left[(x_1-x_4)^2 + (2x_1)^2 + (x_1+x_4)^2\right], \\
{}& \nu(v_6)=6\left[(x_1-x_4)^6 + (2x_1)^6 + (x_1+x_4)^6\right].
\end{align*}
Since $W_{G_2}$ permutes the short roots, the two polynomials above are $W_{G_2}$-invariant.  
  They are a system of fundamental invariant polynomials:
the reason is that the degrees of any two such polynomials are 2 and 6 and 
$\nu(v_6)$ is not a scalar multiple of $(\nu(v_2))^3$. Since otherwise we would have
$$(3x_1^2+x_4^2)^3=r\left[(x_1-x_4)^6+(2x_1)^6+(x_1+x_4)^6\right]$$
for some $r\in \bR$. By making $x_1=0$ one obtains $1=r\cdot 2$ and by making $x_4=0$
one obtains $27= r\cdot 66$, which is a contradiction.
 One now identifies $\mfa=\mfa^*$ and $\mfa'=(\mfa')^*$ via the canonical inner product on $\mfa=\bR^6$ 
 and concludes that point (a) of Proposition \ref{u1u3} holds true.
 
 From now on we regard again $v_2$ and $v_6$ as elements of ${\rm Sym}(\mfa)$ and ${\rm Sym}(\mfa')$.
 That is,
  \begin{align}\label{v2}&v_k:=\left(2\sqrt{\frac{2}{3}}e_6\right)^{k}+\left(\sqrt{\frac{2}{3}}(\sqrt{3}e_5-e_6)\right)^{k}+\left(\sqrt{\frac{2}{3}}(-\sqrt{3}e_5-e_6)\right)^{k}\\{}&
+\sum_{\pm}\left( \pm e_3\pm e_4 -\sqrt{\frac{2}{3}}e_6\right)^k+\sum_{\pm}\left( \pm e_1\pm e_2 -\sqrt{\frac{2}{3}}e_6\right)^k \nonumber \\
{}& +\sum_{\pm}\left(\pm e_2\pm e_4+\frac{1}{\sqrt{6}}(\sqrt{3} e_5+e_6)\right)^k +\sum_{\pm}\left(\pm e_1\pm e_3+\frac{1}{\sqrt{6}}(\sqrt{3} e_5+e_6)\right)^k\nonumber \\{}& +\sum_{\pm}\left(\pm e_2\pm e_3-\frac{1}{\sqrt{6}}(\sqrt{3} e_5-e_6)\right)^k+\sum_{\pm}\left(\pm e_1\pm e_4-\frac{1}{\sqrt{6}}(\sqrt{3} e_5+e_6)\right)^k,\nonumber 
  \end{align}
 for $k=2,6$.

Let  $\Omega_1,\ldots, \Omega_6$ be the elements of
  $U(\mfb)$ described in Theorem \ref{thm:main} for $\Pi$ of type $E_6$ (see { Appendix \ref{trei}}). 
One has
$$\Omega_1=u_1+X_0^2 +\cdots + X_6^2, \quad {\rm where} \ u_1=v_2.$$ 
Observe that
\begin{align*}
{} & X_0^2+ X_1^2+X_5^2= \frac{1}{3}(X_0+X_1+X_5)^2+\frac{1}{3}(X_0-X_1)^2+\frac{1}{3}(X_1-X_5)^2+\frac{1}{3}(X_0-X_5)^2,\\
{} & X_2^2+ X_4^2+X_6^2= \frac{1}{3}(X_2+X_4+X_6)^2+\frac{1}{3}(X_2-X_4)^2+\frac{1}{3}(X_4-X_6)^2+\frac{1}{3}(X_6-X_2)^2.
\end{align*}
\begin{lemma} The differences  $X_0-X_1, X_1-X_5, X_0-X_5, X_2-X_4, X_4-X_6,$ and $X_6-X_2$ are in $\mfb''$.
\end{lemma}
\begin{proof}
Both $X_0+\ve X_1+\ve^2 X_5$ and $X_1+\ve X_0 +\ve^2 X_5$ are in $\mfb''$ and
their difference is $(1-\ve)(X_0-X_1)$. 
\end{proof}
This implies that  
\begin{equation}\label{conclu}\nu(\Omega_1) =\nu(u_1) + (X_0')^2 +  (X_2')^2 + (X_3')^2,\end{equation}
 which  is a multiple of the Laplacian of $\mfb'$.  
 
  Since $\tau: \mfb \to \mfb$ is a Lie algebra homomorphism, it
  induces an algebra homomorphism $\tau : U(\mfb_\bC) \to U(\mfb_\bC)$.

\begin{lemma} \label{abp1}
(a) If $a\in U(\mfb_\bC')$ then $\tau(a)=a$.

(b) If $a\in U(\mfb_\bC)$ has the property that $\tau(a)=\varepsilon a$ or
$\tau(a)=\varepsilon^2 a$ then
$a\in \langle U(\mfb_\bC)\mfb''\rangle$.
\end{lemma}  
 
Proposition \ref{u1u3} can now be proved by the same argument as the one used at the  end of Section \ref{patru}
when proving Proposition \ref{u1u2}.   
 
From Proposition \ref{u1u3} we now deduce:

\begin{corollary}
{ The conclusion of Theorem \ref{thm:main} is} true in the case when $\Pi$ is of type $G_2$ and $\beta$ is the short dominant root. 
\end{corollary}      
     
{ This can be proved in the same way as Corollary \ref{cor:f4}. The only essential difference is that $\Omega_3'$ belongs
to   $U(\mfb'_\bC)_{\rm ev}$ and not necessarily to  $U(\mfb')_{\rm ev}$. However, one can construct out of 
it   an element of $U(\mfb')_{\rm ev}$ whose symbol is $\nu(u_3)$, has degree equal to $6$, and which commutes with  $\Omega'$.
Concretely, consider the basis of $\mfb'$ which consists of $X_0', X_2', X_3'$ along with 
$H_2', H_3'\in \mfa'$ such that $\alpha_i(H_j')=\delta_{ij}$, $i,j=2$ or $3$.  
Relative to the induced PBW basis,  $U(\mfb'_\bC)_{\rm ev}=U(\mfb')_{\rm ev} \bigoplus \sqrt{-1} U(\mfb')_{\rm ev}$.
The aforementioned element of $U(\mfb')_{\rm ev}$ is just the first term in the decomposition of $\Omega_3'$ 
relative to this splitting.}

   \appendix
   
   \section{The case when  $\beta$ is the highest root: the theorem of Etingof}\label{trei}

The complete integrability of the periodic quantum Toda lattice corresponding to a
Dynkin diagram extended by the highest root  was 
proved by Etingof in \cite{etingof}. { We considered necessary to include the details
of his proof for reasons of completeness and clarity. The approach below is slightly 
different from the original one, in that we preferred to use complex Lie groups and loop spaces rather
than formal groups.}
As background references we indicate \cite{Kac}, \cite{Pr-Se}, and \cite{Ne}. 
Let $G$ be a simple,  simply connected, complex Lie group of Lie algebra
$\g$ and  $T\subset G$ a maximal torus, whose Lie algebra we denote by $\h$. Let also $\langle \cdot, \cdot\rangle$ be the  Killing form
on $\g$  normalized such that $\langle\alpha, \alpha\rangle=2$ 
for any long root $\alpha$. 
Pick a simple root system $\alpha_1,\ldots, \alpha_r\in \h^*$ and let $\theta$ be the corresponding  highest root. Consider the differential operator 
on $\h$ given by
$$M=\frac{1}{2}\Delta - K e^{-\theta }- \sum_{i=1}^r e^{\alpha_i}$$  where 
$\Delta$ 
is the Laplacian relative to  $\langle\cdot, \cdot\rangle$ and $K$ is a parameter.  The goal is 
to construct differential operators on $\h$ which commute with $M$.
Recall \cite{Kac} that the corresponding (non-twisted) affine Kac-Moody Lie algebra is $\hat{\g}= {\mathcal L}_{\rm pol}(\g) \oplus \bC c$, 
where ${\mathcal L}_{\rm pol}(\g)=\g\otimes \bC[z, z^{-1}]$ is the space of Laurent polynomials
with coefficients in $\g$,
$z$ being on the unit circle $S^1$.
The Lie bracket on $\hat{\g}$ is defined by
$$[u,v](z)=[u(z), v(z)] + \left({\rm Res} \langle u'(z), v(z)\rangle \right) c$$
for any $u,v\in {\mathcal L}_{\rm pol}(\g)$, where Res stands for residue (the coefficient of $z^{-1}$)
and  $c$ is a central element. 

The Chevalley generators of $\hat{\g}$  are $e_i, h_i, f_i$, where $0\le i \le r$.
Here
$\{h_i\}$ is the  basis of $\hat{\h}:=\h \oplus \bC c$ consisting of the simple affine coroots, $\{e_i\}$ are root vectors for simple affine roots 
and $\{f_i\}$ root vectors for the negatives of those roots. 
Denote by $\n^+$ and $\n^-$ the Lie subalgebras of $\hat{\g}$ generated by $\{e_i\}$ and $\{f_i\}$ respectively.
They can be described as:
$$\n^+=\bigoplus_{\alpha\in \hat{R}^+}\hat{\g}_\alpha,\quad \n^- =\bigoplus_{\alpha\in \hat{R}^-}\hat{\g}_\alpha,$$
where the first sum runs over all positive affine roots and the second sum over all negative affine roots,
$\hat{\g}_\alpha$ being the corresponding root space. 
We thus have the triangular decomposition
$\hat{\g} = \n^+ \oplus \hat{\h} \oplus \n^-.$ 

Let us  now consider  the group ${\mathcal L} (G)$ of all smooth maps from the circle $S^1$ to $G$ along with   
 its universal central extension $\tilde{\mathcal L} (G)$,  see \cite[Section 4.4]{Pr-Se}. Both ${\mathcal L} (G)$
 and $\tilde{\mathcal L} (G)$ are
 Fr\'echet-Lie groups, of  Lie algebras ${\mathcal L} (\g) $ (space of all smooth maps $S^1 \to \g$)
 and ${\mathcal L}(\g)\oplus \bC c  $, respectively. There is a well-defined exponential map
  $\exp : \tilde{\mathcal L} (\g)\to 
 \tilde{\mathcal L} (G)$, which is smooth.

 Let $U^+, U^-\subset G$ be the unipotent radicals of the two standard ``opposite" Borel subgroups
 that contain $T$. Denote by $\cU^{+}$ the subgroup of ${\mathcal L} (G)$ whose elements are
  smooth boundary values of holomorphic maps  $\gamma: \{z\in \bC \mid |z|<1\} \to G$ with  $\gamma(0) \in U^+$. Similarly, $\cU^{-}$ 
 is the subgroup of ${\mathcal L} (G)$ consisting 
 of smooth boundary values of holomorphic maps $\gamma: \{z\in \bC \mid |z|>1\}\cup\{\infty\}  \to G$ such that $\gamma(\infty) \in U^-$.
 There exist canonical embeddings $\cU^{\pm} \subset \tilde{\mathcal L} (G)$. 
 Moreover, if ${\mathcal T}:=\exp (\hat{\h})={ T}\times \bC^*$, then the map   
 $\cU^+ \times {\mathcal T} \times \cU^- \to \tilde{\mathcal L} (G)$, $(u^+, g, u^-)\mapsto u^+ g u^-$ is injective
 onto  $\cU^+  {\mathcal T}  \cU^-$, which is  an open subspace of $\tilde{\mathcal L} (G)$, see \cite[Theorem 8.7.2]{Pr-Se}.
Any element of $\hat{\g}$ induces a canonical left-invariant  vector field on $\tilde{\mathcal L}(G)$, hence also on
its open subspace   $\cU^+ {{\mathcal T}} \cU^-$. 
One obtains in this way a representation of the universal enveloping algebra $U(\hat{\g})$ on  $C^\infty(\cU^+ {{\mathcal T}} \cU^-,\bC)$.
Concretely, for $\xi \in \hat{\g}$,  $u^\pm \in {\mathcal U}^\pm$, $g\in {\mathcal T}$, and $f : \cU^+ {{\mathcal T}} \cU^-\to \bC$ smooth,
\begin{equation}\label{xif}(\xi.f) (u^+ g u^-)=\frac{d}{dt}\Big|_{t=0} f(\exp (t\xi) u^+gu^-).\end{equation}

Note that $\n^+ \oplus \h \oplus \n^-={\mathcal L}_{\rm pol}(\g)$, which is a dense subspace 
of ${\mathcal L}(\g)$ relative to the Fr\'echet topology.  
Moreover, the Lie algebra of $\cU^{\pm}$ is the closure of $\n^{\pm}$ in
${\mathcal L}(\g)$ relative to the aforementioned topology.
Recall that $\n^+$ is generated as a Lie algebra by $e_i$, $0\le i \le r$. Pick some complex numbers
   $ \chi^+_0, \ldots, \chi^+_r$ and consider the Lie algebra homomorphism $\chi^{+}:\n^+ \to \bC$ determined by
   $\chi^+(e_i)=\chi_i^+$, for all $0\le i \le r.$ Since  $\cU^{+}$ is simply connected  and
   its Lie algebra contains $\n^+$ as a dense subspace, there is a unique  lift $\chi^+: \cU^{+} \to \bC^*$    (which is a group homomorphism). In the same way, by picking  $ \chi^-_0, \ldots, \chi^-_r\in \bC$, one attaches
   to them a Lie algebra homomorphism $\chi^{-}:\n^- \to \bC$ such that 
   $\chi^-(f_i)=\chi_i^-$, $0\le i \le r$, and then a Lie group homomorphism $\chi^-: \cU^{-} \to \bC^*$. 
  A smooth  function $\phi: \cU^+ {{\mathcal T}} \cU^- \to \ \bC$ is called a {\it Whittaker function} if:
  \begin{align}
  \label{phiu}{}&\phi(u_+ g u_-) = \chi^+(u_+) \phi(g) \chi^-(u_-), \ {\rm for \  all} \ u_{\pm}\in \cU^{\pm} \ {\rm and} \ g\in {\mathcal T},\\
  {}& c.\phi=-h^\vee \phi,
  \end{align}
  where $h^\vee$ is the dual Coxeter number of $\g$. Such a function $\phi$ is 
clearly determined by its restriction to ${\mathcal T}$. In turn, for any $g\in T$ and any 
$t_0\in \bR$, 
$$\frac{d}{dt}\Big|_{t=t_0}\phi(\exp(t c) g) =-h^\vee  \phi(\exp(t_0c)g).$$
The initial value problem formed by this equation along with the
condition $\phi(\exp(tc)g)|_{t=0}=\phi(g)$ has a unique solution.
Thus, $\phi$ 
 is uniquely determined by its    
   values on
  $T$, hence the space of Whittaker functions can be naturally identified with $C^\infty(T, \bC)$. 
  
  Equation (\ref{phiu}) implies readily that for any Whittaker function $\phi$ and any $0\le i, j\le r$ one has:
  \begin{align}
 {}& e_i.\phi = \chi^+_i \phi \ {\rm on } \ {\mathcal T}\label{e}\\
 {}& f_i.\phi  
=  \chi_i^- e^{\alpha_i} \phi \ {\rm on } \ {\mathcal T}\\
{}& f_i.e_i.\phi = \chi^-_i \chi^+_i e^{\alpha_i}\phi \ {\rm on } \ {\mathcal T}\\
{}& f_i.f_j.\phi = \chi^-_i \chi^-_j e^{\alpha_i}e^{\alpha_j}\phi \ {\rm on } \ {\mathcal T}\\
{}& {\rm if} \  \alpha \ {\rm  is \  an \  affine \  positive \ nonsimple \   root \  and \ } \xi\in \hat{\g}_\alpha \ {\rm then \ }
  \xi.\phi =0 \ {\rm on } \ {\mathcal T}\label{ff}.
  \end{align}

    There exists a certain completion $\tilde{U}(\hat{\g})$ of $U(\hat{\g})$ such that the  center of $\tilde{U}(\hat{\g})/(c+h^\vee)$ 
    is rich, see \cite{FF, F}.
    First of all, it contains a degree two Casimir element ${C}$, which can be expressed as:
   \begin{equation}\label{hatc}{C} =\sum_{i=1}^r \xi_i^2 +2h_\rho +2\sum_{\alpha \in \hat{R}^+}\left(\sum_k f^k_\alpha e^k_\alpha\right).\end{equation}
    Here $\xi_1,\ldots, \xi_r$ is an orthonormal basis of $\h$, and for each positive affine root $\alpha \in \hat{R}^+$,
    the vectors  $e_\alpha^k$ are a basis of $\hat{\g}_\alpha$ and $ f_\alpha^k$ a basis of $\hat{\g}_{-\alpha}$
  such that $\langle e_\alpha^k,f_\alpha^\ell\rangle=\delta_{k\ell}$. For $\alpha = \alpha_i$ both $\hat{\g}_{\alpha}$ and 
  $\hat{\g}_{-\alpha}$ are
  one dimensional and we take $e_{\alpha_i}^1=e_i$, $f_{\alpha_i}^1=f_i$, $0\le i \le r$. Finally, $\rho$ denotes the half-sum of all positive
   roots of $(\g, \h)$ relative to the basis $\alpha_1,\ldots, \alpha_r$ and $h_\rho$ the element of $\h$ that corresponds to it  via the Killing form. Note that even though the second sum in
  (\ref{hatc}) is infinite, $C$ belongs to the completion $\tilde{U}(\hat{\g})$. 
  
  To any $\xi \in \mft$ we attach the directional derivative  $\partial_\xi$ on $C^\infty(T,\bC)$ given by
  $$(\partial_\xi f)(g) =\frac{d}{dt}\Big|_{t=0}f(\exp(t\xi)g)$$
  for all $f \in C^\infty(T, \bC)$ and all $g\in T$. 
Equations (\ref{e})-(\ref{hatc}) above imply that for any Whittaker function $\phi$ we have
  $(C.\phi)|_T = D (\phi|_T)$, where $D$ is the differential operator on $C^\infty(T,\bC)$ given by
  $$D=\sum_{i=1}^r \partial_{\xi_i}^2 + 2\partial_{h_\rho} + 2\sum_{i=0}^r \chi^+_i \chi^-_i e^{\alpha_i}.$$
  Here  $\partial_{\xi_1},\ldots, \partial_{\xi_r}, \partial_{h_\rho}$ are derivatives and 
  $e^{\alpha_i}$ is the function $T\to \bC$ given by $\exp(h)\mapsto e^{\alpha_i(h)}$, for all $h\in \h$, $0\le i \le r$.
  Observe now that $h_\rho = \sum_{i=1}^r \rho(\xi_i) \xi_i$, hence the composition of $D$ with $e^{-\rho}$ is
 \begin{align*}{}& De^{-\rho}=\left(\sum_{i=1}^r  \partial_{\xi_i}^2 + 2\partial_{h_\rho}+ 2\sum_{i=0}^r \chi^+_i \chi^-_i e^{\alpha_i}\right)e^{-\rho}\\
 {}& = \sum_{i=1}^r e^{-\rho} (\rho(\xi_i)^2-2\rho(\xi_i)\partial_{\xi_i} +\partial_{\xi_i}^2) 
 +2\sum_{i=1}^r\rho(\xi_i) e^{-\rho}(-\rho(\xi_i)+\partial_{\xi_i}) + 2\sum_{i=0}^r \chi^+_i \chi^-_i e^{\alpha_i}e^{-\rho}\\
 {}& = e^{-\rho}\left(\sum_{i=1}^r \partial_{\xi_i}^2 - \sum_{i=1}^r \rho(\xi_i)^2 + 2\sum_{i=0}^r \chi^+_i \chi^-_i e^{\alpha_i}\right).
 \end{align*}
  Since $\sum_{i=1}^r \rho(\xi_i)^2=\langle \rho, \rho\rangle$ and $\alpha_0 =-\theta$ on $\h$, one obtains
  $$e^{\rho} D e^{-\rho} = \sum_{i=1}^r \partial_{\xi_i}^2  +2\chi^+_0 \chi^-_0 e^{-\theta} +2\sum_{i=1}^r \chi^+_i \chi^-_i e^{\alpha_i} -\langle \rho,\rho\rangle.$$
  Choose $\chi_i^+$ and $\chi_i^-$  such that $\chi_i^+\chi_i^-=-1$ for $i\neq 1$,  $\chi_0^+=-1$, and $\chi_0^- =K$. Then 
  $$M=\frac{1}{2}(e^{\rho} D e^{-\rho} +\langle \rho,\rho\rangle).$$

    The center of $\tilde{U}({\g'}_\af)/(c+h^\vee)$ contains  $Y_1:={C}$ along with $Y_2, \ldots, Y_r$, which are of the form 
    $Y_i=u_i + Y_i^+$, where $u_i\in {\rm Sym}(\h)$ are  fundamental generators
    of ${\rm Sym}(\h)^W$ and $Y_i^+$ is a sum of monomials in $U(\hat{\g})\n^+$ of degree at most equal to $\deg u_i$,
    see \cite{FF, F}.  
From equations (\ref{e})-(\ref{ff}) one deduces that for any Whittaker function $\phi$ one has $(Y_i.\phi)|_T =D_i(\phi|_T)$, $1\le i \le r$,
where $D_i$ is a differential operator on $C^\infty(T,\bC)$ which admits a presentation as a  polynomial in
$ Ke^{-\theta}, e^{\alpha_1}, \ldots, e^{\alpha_r},$ $ \partial_{\xi_1}, \ldots, \partial_{\xi_r}$
(one also uses that $(f_0.\phi)|_T= \chi_0^-e^{\alpha_0}\phi|_T = Ke^{-\theta}\phi|_T).$ Furthermore, the  symbol of $D_i$ is $u_i(\partial_{\xi_1}, \ldots, \partial_{\xi_r})$.
    The differential operators  $D=D_1, D_2,  \ldots, D_r$ commute with each other.
    But then also $$D'_1:=M= \frac{1}{2}(e^{\rho} D e^{-\rho} +\langle \rho,\rho\rangle), D_2':=e^{\rho} D_2 e^{-\rho},\ldots,
  D_r':=e^{\rho} D_r e^{-\rho}$$
  commute with each other. Each of them is a  polynomial in
$ Ke^{-\theta}, e^{\alpha_1}, \ldots, e^{\alpha_r},$ $ \partial_{\xi_1}, \ldots, \partial_{\xi_r}$,
since  $e^\rho \partial_{\xi_i} e^{-\rho} = - \rho(\xi_i) + \partial_{\xi_i}$.
Moreover, the symbol of $D_i'$ is $u_i(\partial_{\xi_1}, \ldots, \partial_{\xi_r})$ as well. 

Let us now consider the  $``ax+b"$ algebra $\mfb=\mfa\oplus \mfu$ corresponding to the Dynkin diagram of $\g$
extended with the highest root. 
Its complexification is $\mfb_\bC=\mfa_\bC \oplus \mfu_\bC$, where one can identify $\mfa_\bC=\h$.
Then  
$\mfu_\bC$ has a $\bC$-basis  $X_0,\ldots, X_r$ such that the Lie bracket $[ \ , \ ]$ on $\mfb_\bC$ is
identically zero on both $\h$ and $\mfu_\bC$ and for any $h\in \h$ one has 
\begin{align*}
{}&[h, X_i]=\alpha_i(h)X_i, \quad 1\le i \le r\\
{}&[h, X_0]=-\theta(h)X_0.
\end{align*}
Let  $\Omega_1,\ldots, \Omega_r$ be the elements of $U(\mfb_\bC)$ which are obtained from the differential operators $D'_1, \ldots, D'_r$ above by making the following replacements:
    \begin{align*}
    e^{\alpha_i/2}\mapsto \frac{\sqrt{-1}}{2\sqrt{2}} X_{i},\quad \sqrt{K} e^{-\theta/2}\mapsto \frac{\sqrt{-1}}{2\sqrt{2}} X_{0},
    \quad \partial_{H_i} \mapsto \frac{1}{2}H_i; \quad 1\le i \le r.
    \end{align*}
In this way $$M=\frac{1}{2}\sum_{i,j=1}^r \langle\alpha_i,\alpha_j\rangle \partial_{H_i} \partial_{H_j}
     -Ke^{-\theta}-\sum_{i=1}^r e^{\alpha_i}$$
     turns into $\frac{1}{8}\Omega$, see equation (\ref{laplace}). Moreover, if $\theta=m_1\alpha_1+\cdots + m_r\alpha_r$, then for any $1\le i,j \le r$ we have:
     \begin{align*}{}& \left[\frac{1}{2}H_i, \frac{\sqrt{-1}}{2\sqrt{2}}X_j\right]=
     \delta_{ij}\frac{1}{2}\cdot\frac{\sqrt{-1}}{2\sqrt{2}}X_j, \quad
     \left[\frac{1}{2}H_i, \frac{\sqrt{-1}}{2\sqrt{2}}X_{0}\right]=
    - m_i \cdot \frac{1}{2}\cdot\frac{\sqrt{-1}}{2\sqrt{2}}X_{0}\\ 
     {}&\left[ \partial_{H_i}, e^{\alpha_j/2}\right] =\delta_{ij}\frac{1}{2}e^{\alpha_j/2}, 
     \quad    \left[ \partial_{H_i}, e^{-\theta/2}\right]=-m_i\cdot \frac{1}{2} e^{-\theta/2}.  
     \end{align*}
Thus $\Omega_i$ is in $U(\mfb_\bC)_\eve$ and satisfies the three conditions in Theorem \ref{thm:main}.

\begin{proposition}\label{long}
{ The conclusion of Theorem \ref{thm:main} is} true in the case when { $\Pi$ is an arbitrary 
irreducible root system}  
and $\beta$ is the highest root.  
\end{proposition}

   \begin{proof}
   The operators $\Omega_1, \ldots, \Omega_r$ satisfy the conditions in Theorem \ref{thm:main}.
   However, we only know  that $\Omega_2, \ldots, \Omega_r$  are in $U(\mfb_\bC)$, although not necessarily in $U(\mfb)$. 
   For $i\in \{2, \ldots, r\}$ one considers the expansion of $\Omega_i$ relative to the PBW basis $X^IH^J$.
   By writing each coefficient as $a+\sqrt{-1}b$, with $a, b\in \bR$, one  obtains     $\Omega_i = \Omega_i' + \sqrt{-1}\Omega_i''$, where both $ \Omega_i'$ and $\Omega_i''$ are in 
    $U(\mfb)$.  But then $[\Omega_i'+\sqrt{-1}\Omega_i'', \Omega]=[\Omega_i', \Omega]+\sqrt{-1}[\Omega_i'', \Omega]$, which is 
    equal to 0. This implies that $[\Omega_i', \Omega]=0.$ Thus $\Omega'_2, \ldots, \Omega'_r$  are in $U(\mfb)$ and satisfy
    the three conditions in Theorem \ref{thm:main}. It remains to justify the last statement in the theorem,
    that is, that $[\Omega_i', \Omega_j']=0$, for all $2\le i, j\le r$.  This can be proved as follows. 
    First note that, since both $\Omega_i'$ and  $\Omega_j'$ are in
    $U(\mfb)_\eve$, their bracket $[\Omega_i', \Omega_j']$ is in $U(\mfb)_\eve$ as well.
    Also, since both $\Omega_i'$ and  $\Omega_j'$ commute with $\Omega$, 
   their bracket $[\Omega_i', \Omega_j']$ commutes with $\Omega$ as well. 
   Since $\mu([\Omega_i', \Omega_j'])=[\mu(\Omega_i'), \mu(\Omega_j')]=0$, one can use Proposition \ref{degree}  
   to deduce that $[\Omega_i', \Omega_j']$ is a multiple of $(X_0X_1^{m_1} \cdots X_r^{m_r})^2$.
On the other hand the degree of   $[\Omega_i', \Omega_j']$ is at most equal to $\deg \Omega_i'+\deg \Omega_j' -1$, which is
the same as $\deg u_i +\deg u_j -1$. { The table in Section \ref{doi} shows that 
 $\deg u_i +\deg u_j -1$  is strictly smaller that $2(1+m_1 + \cdots + m_r)$, for all $1\le i<j\le r$
 (recall that $m_1,\ldots, m_r$ are the coefficients of the highest root expansion).}
We conclude that  $[\Omega_i', \Omega_j']=0$, as required.
   \end{proof}

\end{document}